\documentclass[final]{siamltex}

\usepackage{graphicx}
\usepackage{color}
\usepackage{bm}
\usepackage{amsfonts}

\newcommand{\tr}{^{\sf T}}
\newcommand{\m}[1]{{\bf{#1}}}
\newcommand{\g}[1]{\mbox{\boldmath $#1$}}
\newcommand{\C}[1]{{\cal {#1}}}
\input epsf

\title{
An Active Set Algorithm for Nonlinear Optimization with Polyhedral Constraints
\thanks{
February 28, 2016.
The authors gratefully acknowledge support by the National
Science Foundation under grants 1522629 and 1522654,
by the Office of Naval Research under grants
N00014-11-1-0068 and N00014-15-1-2048, by the Air Force Research
Laboratory under contract FA8651-08-D-0108/0054, and by
the National Science Foundation of China under grant 11571178.
}
}

\author{
    William W. Hager\thanks{{\tt hager@ufl.edu},
        http://people.clas.ufl.edu/hager/,
        PO Box 118105,
        Department of Mathematics,
        University of Florida, Gainesville, FL 32611-8105.
        Phone (352) 294-2308. Fax (352) 392-8357.}
\and
    Hongchao Zhang\thanks{{\tt hozhang@math.lsu.edu},
        http://www.math.lsu.edu/$\sim$hozhang,
        Department of Mathematics,
        Louisiana State University, Baton Rouge, LA 70803-4918.
        Phone (225) 578-1982. Fax (225) 578-4276.}
}

\begin{document}

\maketitle

\begin{abstract}
A polyhedral active set algorithm PASA is developed for solving a
nonlinear optimization problem whose feasible set is a polyhedron.
Phase one of the algorithm is the gradient projection method,
while phase two is any algorithm for solving a linearly constrained
optimization problem.
Rules are provided for branching between the two phases.
Global convergence to a stationary point is established,
while asymptotically PASA performs
only phase two when either a nondegeneracy assumption holds,
or the active constraints are linearly independent and a strong
second-order sufficient optimality condition holds.
\end{abstract}

\begin{keywords}
polyhedral constrained optimization,
active set algorithm, PASA, gradient projection algorithm,
local and global convergence
\end{keywords}

\begin{AMS}
90C06, 90C26, 65Y20
\end{AMS}

\pagestyle{myheadings}
\thispagestyle{plain}
\markboth{W. W. HAGER and H. ZHANG}
{POLYHEDRAL CONSTRAINED NONLINEAR OPTIMIZATION}

\section{Introduction}

%
We develop an active set algorithm for a general
nonlinear polyhedral constrained optimization problem
\begin{equation} \label{P}
\min \; \{ f (\m{x}) : \m{x} \in \Omega \}, \quad
\mbox{where }
\Omega = \{ \m{x} \in \mathbb{R}^n: \m{Ax} \le \m{b} \}.
\end{equation}
Here $f$ is a real-valued, continuously differentiable function,
$\m{A} \in \mathbb{R}^{m \times n}$, $\m{b} \in \mathbb{R}^m$, and
$\Omega$ is assumed to be nonempty.
In an earlier paper \cite{hz05a}, we developed an active set algorithm
for bound constrained optimization.
In this paper, we develop new machinery for handling the more
complex polyhedral constraints of (\ref{P}).
Our polyhedral active set algorithm (PASA) has two phases:
phase one is the gradient projection algorithm,
while phase two is any algorithm for solving a linearly constrained
optimization problem over a face of the polyhedron $\Omega$.
The gradient projection algorithm of phase one
is robust in the sense that it converges
to a stationary point under mild assumptions, but the convergence rate
is often linear at best.
When optimizing over a face of the polyhedron in phase two,
we could accelerate the convergence through the use of a superlinearly
convergent algorithm based on conjugate gradients,
a quasi-Newton update, or a Newton iteration.
In this paper, we give rules for switching between phases
which ensure that asymptotically, only phase two is performed.
Hence, the asymptotic convergence rate of PASA coincides with the
convergence rate of the scheme used to solve the linearly constrained
problem of phase two.
A separate paper will focus on a specific numerical implementation of PASA.

We briefly survey some of the rich history of active set methods.
Some of the initial work focused on
the use of the conjugate gradient method with bound constraints
as in \cite{dt83, do97, do03, dfs03, mt91, po69, yt91}.
Work on gradient projection methods include
\cite{be76, cm87, G64,lp66, mt72, sp97}.
Convergence is accelerated using Newton and trust region
methods \cite{cgt00}.
Superlinear and quadratic convergence for nondegenerate problems
can be found in \cite{be82,bmt90, cgt88, fjs98}, while
analogous convergence results are given in
\cite{flp02, fms94, le91, lm99}, even for degenerate problems.
The affine scaling interior point approach
\cite{bcl99, cl94, cl96, cl97, dhv98, HagerZhang13a, huu99, kk05, uuh99, zh04}
is related to the trust region algorithm.
Linear, superlinear, and quadratic convergence results have been established.

Recent developments on active set methods for quadratic programming problem
can be found in \cite{ForsgrenGillWong15, GillWong15}.
A treatment of active set methods in a rather general setting is
given in \cite{IzmailovSolodov14}.
We also point out the recent work \cite{GoldbergLeyffer15} on
a very efficient two-phase
active set method for conic-constrained quadratic programming, and
the earlier work \cite{FriedlanderLeyffer08} on a two-phase
active set method for quadratic programming.
As in \cite{hz05a} the first phase in both applications
is the gradient projection method.
The second phase is a Newton method in \cite{GoldbergLeyffer15},
while it is a linear solver in \cite{FriedlanderLeyffer08}.
Note that PASA applies to the general nonlinear objective in (\ref{P}).
Active set strategies were applied to
$\ell_1$ minimization in \cite{wyzg12,wygz10} and in \cite{hz15}
they were applied to the minimization of a nonsmooth dual
problem that arises when projecting a point onto a polyhedron.
In \cite{wyzg12,wygz10}, a nonmonotone line search based on
``Shrinkage" is used to estimate a support at the solution,
while a nonmonotone SpaRSA algorithm \cite{HagerPhanZhang11,SpaRSA} is used
in \cite{hz15} to approximately identify active constraints.

Unlike most active set methods in the literature,
our algorithm is not guaranteed to identify the active constraints in
a finite number of iterations due to the structure of the line
search in the gradient projection phase.
Instead, we show that only the fast phase two algorithm is performed
asymptotically, even when strict complementary slackness is violated.
Moreover, our line search only requires one projection in
each iteration, while algorithms that identify active constraints often
employ a piecewise projection scheme that may require
additional projections when the stepsize increases.

The paper is organized as follows.
Section~\ref{structure} gives a detailed statement of the polyhedral
active set algorithm, while Section~\ref{global} establishes its global
convergence.
Section~\ref{stationarity} gives some properties for the solution
and multipliers associated with a Euclidean projection onto $\Omega$.
Finally, Section~\ref{nondegenerate} shows that asymptotically
PASA performs only phase two when converging to a nondegenerate
stationary point, while Section~\ref{degenerate} establishes the
analogous result for degenerate problems when the active constraint
gradients are linearly independent and a strong second-order sufficient
optimality condition holds.

{\bf Notation.}
Throughout the paper, $c$ denotes a generic nonnegative constant which
has different values in different inequalities.
For any set $\C{S}$, $|\C{S} |$ stands for the number of elements
(cardinality) of $\C{S}$, while $\C{S}^c$ is the complement of $\C{S}$.
The set $\C{S} - \m{x}$ is defined by
\[
\C{S} - \m{x} = \{ \m{y} - \m{x} : \m{y} \in \C{S} \}.
\]
The distance between a set $\C{S} \subset \mathbb{R}^n$ and a point
$\m{x} \in \mathbb{R}^n$ is given by
\[
\mbox{dist }(\m{x}, \C{S}) = \inf \{ \|\m{x} - \m{y}\|: \m{y} \in \C{S} \},
\]
%
where $\| \cdot \|$ is the Euclidean norm.
The subscript $k$ is often used to denote the iteration number
in an algorithm, while $x_{ki}$ stands for the $i$-th component of
the iterate $\m{x}_k$.
The gradient $\nabla f(\m{x})$ is a row vector while
$\m{g}(\m{x}) = \nabla f(\m{x})\tr$ is the gradient arranged as a column vector;
here $\tr$ denotes transpose.
The gradient at the iterate $\m{x}_k$ is $\m{g}_k = \m{g}(\m{x}_k)$.
In several theorems, we assume that $f$ is Lipschitz continuously differentiable
in a neighborhood of a stationary point $\m{x}^*$.
The Lipschitz constant for $\nabla f$ is always denoted $\kappa$.
We let $\nabla^2 f(\m{x})$ denote the Hessian of $f$ at $\m{x}$.
The ball with center $\m{x}$ and radius $r$ is denoted
$\C{B}_{r}(\m{x})$.
For any matrix $\m{M}$, $\C{N}(\m{M})$ is the null space.
If $\C{S}$ is a subset of the row indices of $\m{M}$,
then $\m{M}_{\C{S}}$ denotes the submatrix of $\m{M}$ with row indices $\C{S}$.
For any vector $\m{b}$, $\m{b}_{\C{S}}$
is the subvector of $\m{b}$ with indices  $\C{S}$.
$P_\Omega (\m{x})$ denotes the Euclidean projection of $\m{x}$ onto $\Omega$:
\begin{equation} \label{proj}
\C{P}_{\Omega} (\m{x}) =
\arg  \; \min \{ \|\m{x} - \m{y}\| : \m{y} \in \Omega \} .
\end{equation}
For any $\m{x} \in \Omega$, the active and free index sets are
defined by
\[
\C{A}(\m{x}) = \{ i: (\m{Ax} - \m{b})_i = 0\} \quad \mbox{and} \quad
\C{F}(\m{x}) = \{ i: (\m{Ax} - \m{b})_i < 0\},
\]
respectively.

\section{Structure of the algorithm}
\label{structure}
As explained in the introduction,
PASA uses the gradient projection algorithm in phase one and
a linearly constrained optimization algorithm in phase two.
Algorithm~\ref{GPA} is the gradient projection algorithm (GPA) used for the
analysis in this paper.
A cartoon of the algorithm appears in Figure~\ref{gpa-fig}.
\renewcommand\figurename{Alg.}
\begin{figure}[t]
\begin{center}
{ \tt
\begin{tabular}{l}
\hline
\\
{\bf \textcolor{blue}{Parameters:}}
$\delta$ and $\eta \in (0, 1)$, $\alpha \in (0, \infty)$\\[.05in]
While stopping condition does not hold\\[.05in]
\hspace*{.3in}1.~$\m{d}_k =$
$\m{y}(\m{x}_k, \alpha) - \m{x}_k$, $\m{y}(\m{x}, \alpha) =$
$\C{P}_\Omega (\m{x} - \alpha \m{g}(\m{x}))$ \\[.05in]
\hspace*{.3in}2.~$s_k = \eta^j$ where $j \ge 0$ is smallest integer
such that \\[.05in]
\hspace*{.6in}$f(\m{x}_k + s_k \m{d}_k) \le f(\m{x}_k) + s_k \delta
\nabla f(\m{x}_k)\m{d}_k$\\[.05in]
\hspace*{.3in}3.~$\m{x}_{k+1} = \m{x}_k + s_k \m{d}_k$
and $k \leftarrow k + 1$
\\[.05in]
End while\\
\hline
\end{tabular}
}
\end{center}
\caption{Prototype gradient projection algorithm (GPA). \label{GPA}}
\end{figure}
\renewcommand\figurename{Fig.}
\begin{figure}
\begin{center}
\includegraphics[scale=0.5]{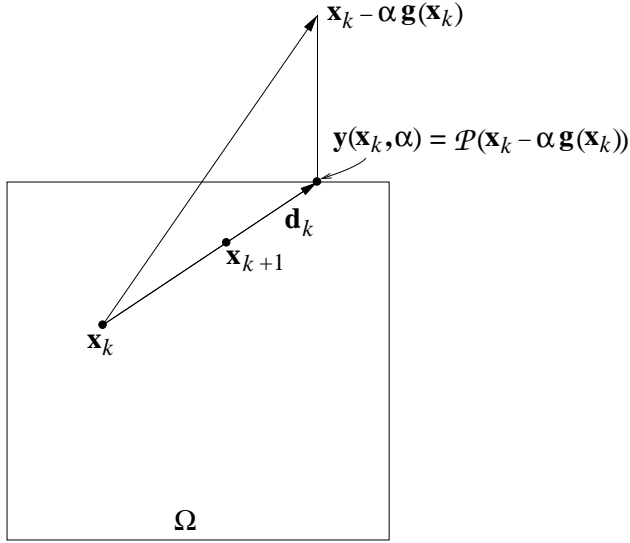}
\caption{An iteration of the gradient projection algorithm.
\label{gpa-fig}}
\end{center}
\end{figure}
This is a simple monotone algorithm based on an Armijo line search.
Better numerical performance is achieved with a more general nonmonotone
line search such as that given in \cite{hz05a},
and all the analysis directly extends to this more general framework;
however, to simplify the analysis and discussion in the paper, we utilize
Algorithm~\ref{GPA} for the GPA.

The requirements for the linearly constrained optimizer (LCO)
of phase two, which operates on the {\it faces} of $\Omega$, are now developed.
One of the requirements is that when the active sets repeat in an
infinite series of iterations, then the iterates must
approach stationary.
To formulate this requirement in a precise way,
we define
\begin{equation}\label{dL}
\m{g}^\C{I} (\m{x}) =
\C{P}_{\C{N}(\m{A}_\C{I})}(\m{g}(\m{x})) =
\arg \; \min \{ \|\m{y} - \m{g}(\m{x})\| : \m{y} \in \mathbb{R}^n \mbox{ and }
\m{A}_\C{I} \m{y} = \m{0} \}.
\end{equation}
Thus $\m{g}^\C{I} (\m{x})$ is the projection of the gradient
$\m{g}(\m{x})$ onto the null space $\C{N}(\m{A}_\C{I})$.
We also let $\m{g}^\C{A}(\m{x})$ denote
$\m{g}^\C{I} (\m{x})$ for $\C{I} = \C{A}(\m{x})$.
If $\C{A}(\m{x})$ is empty, then $\m{g}^\C{A} (\m{x}) = \m{g}(\m{x})$,
while if $\m{x}$ is a vertex of $\Omega$, then $\m{g}^\C{A} (\m{x}) = \m{0}$.
This suggests that $e(\m{x}) = \|\m{g}^\C{A} (\m{x})\|$ represents a
local measure of stationarity in the sense
that it vanishes if and only if $\m{x}$ is a stationary point on its
associated face
\[
\{ \m{y} \in \Omega : (\m{Ay} - \m{b})_i = 0
\mbox{ for all } i \in \C{A}(\m{x}) \} .
\]
The requirements for the phase two LCO are the following:
\smallskip

\begin{itemize}
\item[F1.]
$\m{x}_{k} \in \Omega$ and $f(\m{x}_{k+1}) \le f(\m{x}_k)$ for each $k$.
\item[F2.]
$\C{A}(\m{x}_{k}) \subset \C{A}(\m{x}_{k+1})$ for each $k$.
\item[F3.]
If $\C{A}(\m{x}_{j+1}) = \C{A}(\m{x}_j)$
for $j \ge k$, then $\liminf\limits_{j\to\infty} e(\m{x}_j) = 0$.
\end{itemize}
\smallskip

Condition F1 requires that the iterates in phase two are monotone,
in contrast to phase one where the iterates could be nonmonotone.
By F2 the active set only grows during phase two, while F3
implies that the local stationarity measure becomes small
when the active set does not change.
Conditions F1, F2, and F3 are easily fulfilled by algorithms based
on gradient or Newton type iterations which employ a monotone line search
and which add constraints to the active set whenever a new constraint
becomes active.

Our decision for switching between phase one (GPA)
and phase two (LCO) is based on a
comparison of two different measures of stationarity.
One measure is the local stationarity measure $e(\cdot)$, introduced already,
which measures stationarity relative to a face of $\Omega$.
The second measure of stationarity is a global metric in the sense that
it vanishes at $\m{x}$ if and only if $\m{x}$ is
a stationary point for the optimization problem (\ref{P}).
For $\alpha \ge 0$, let $\m{y}(\m{x}, \alpha)$
be the point obtained by taking a step from $\m{x}$ along the negative
gradient and projecting onto $\Omega$; that is,
\begin{equation}\label{y-def}
\m{y}(\m{x}, \alpha) = \C{P}_\Omega (\m{x} - \alpha \m{g}(\m{x})) =
\arg \;
\min \left\{ \frac{1}{2} \|\m{x} - \alpha \m{g}(\m{x}) - \m{y}\|^2:
\m{Ay} \le \m{b} \right\}.
\end{equation}
The vector
\begin{equation}\label{dalpha}
\m{d}^\alpha (\m{x}) = \m{y}(\m{x},\alpha) - \m{x}
\end{equation}
points from $\m{x}$ to the projection of
$\m{x} - \alpha \m{g}(\m{x})$ onto $\Omega$.
As seen in \cite[Prop. 2.1]{hz05a}, when $\alpha > 0$,
$\m{d}^\alpha (\m{x}) = \m{0}$
if and only if $\m{x}$ is a stationary point for (\ref{P}).
We monitor convergence to a stationary point using the function $E$
defined by
\[
E(\m{x}) = \|\m{d}^1(\m{x}) \|.
\]
$E(\m{x})$ vanishes if and only if $\m{x}$ is a stationary point of (\ref{P}).
If $\Omega = \mathbb{R}^n$, then $E(\m{x}) = \|\m{g}(\m{x})\|$, the
norm of the gradient, which is the usual way to assess convergence
to a stationary point in unconstrained optimization.

The rules for switching between phase one and phase two depend on the relative
size of the stationarity measures $E$ and $e$.
We choose a parameter $\theta \in (0, 1)$ and branch from phase one to
phase two when $e(\m{x}_k) \ge \theta E(\m{x}_k)$.
Similarly, we branch from phase two to phase one when
$e(\m{x}_k) < \theta E(\m{x}_k)$.
To ensure that only phase two is executed asymptotically at a degenerate
stationary point, we may need to decrease $\theta$ as the iterates converge.
The decision to decrease $\theta$ is based on what we called the
undecided index set $\C{U}$ which is defined as follows.
Let $\m{x}$ denote the current iterate,
let $E(\m{x})$ be the global measure of stationarity,
and let $\g{\lambda}(\m{x})$ denote any Lagrange multiplier associated with
the polyhedral constraint in (\ref{y-def}) and $\alpha = 1$.
That is, if $\m{y} = \m{y}(\m{x}, 1)$ is the solution of (\ref{y-def})
for $\alpha = 1$, then $\g{\lambda} (\m{x})$ is any vector
that satisfies the conditions
\[
\m{y} - \m{x} + \m{g}(\m{x}) + \m{A}\tr \g{\lambda} (\m{x}) = \m{0}, \quad
\g{\lambda} (\m{x}) \ge \m{0}, \quad \lambda_i (\m{x}) = 0
\mbox{ if } i \in \C{F}(\m{y}).
\]
Given parameters $\beta \in (1, 2)$ and $\gamma \in (0,1)$,
the undecided index set is defined by
\[
\C{U}(\m{x}) = \{ i : \lambda_i (\m{x}) \ge E(\m{x})^\gamma
\mbox{ and } (\m{b} - \m{Ax})_i \ge E(\m{x})^\beta \} .
\]

If $\m{x}$ is close enough to a stationary point that $E(\m{x})$ is small,
then the indices in $\C{U}(\m{x})$ correspond to those constraints
for which the associated multiplier $\lambda_i (\m{x})$ is relatively
large in the sense that $\lambda_i (\m{x}) \ge E(\m{x})^\gamma$,
and the $i$-th constraint is relatively inactive in the sense
that $(\m{b} - \m{Ay})_i \ge E(\m{x})^\beta$.
By the first-order optimality conditions at a local minimizer,
large multipliers are associated with active constraints.
Hence, when the multiplier is relatively large and the constraint is
relatively inactive, we consider the constraint undecided.
When $\C{U}(\m{x})$ is empty, then we feel that the active constraints
are nearly identified, so we decrease $\theta$ in phase one so that
phase two will compute a more accurate local stationary point before
branching back to phase one.

Algorithm~\ref{PASA} is the polyhedral active set algorithm (PASA).
The parameter $\epsilon$ is the convergence tolerance,
the parameter $\theta$ controls the branching between phase one and
phase two, while the parameter $\mu$ controls the decay of $\theta$
when the undecided index set is empty.
\renewcommand\figurename{Alg.}
\begin{figure}[t]
\begin{center}
{ \tt
\begin{tabular}{l}
\hline
\\
{\bf \textcolor{blue}{Parameters:}}
$\epsilon \in [0, \infty)$, $\theta$ and $\mu \in (0, 1)$\\[.05in]
$\m{x}_1 = \C{P}_\Omega (\m{x}_0)$, $k = 1$ \\[.05in]
{\bf \textcolor{blue}{Phase one:}}
While $E(\m{x}_k) > \epsilon$ execute GPA\\[.05in]
\hspace*{.3in}If $\C{U}(\m{x}_k) = \emptyset$ and
$e(\m{x}_k) < \theta E(\m{x}_k)$, then $\theta \leftarrow \mu \theta$.\\[.05in]
\hspace*{.3in}If $e(\m{x}_k) \ge \theta E(\m{x}_k)$, goto phase two.
\\[.05in]
End while\\[.2in]
{\bf \textcolor{blue}{Phase two:}}
While $E(\m{x}_k) > \epsilon$ execute LCO \\[.05in]
\hspace*{.3in}If $e(\m{x}_k) < \theta E(\m{x}_k)$, goto phase one.\\[.05in]
End while\\
\hline
\end{tabular}
}
\end{center}
\caption{Polyhedral active set algorithm (PASA). \label{PASA}}
\end{figure}
\renewcommand\figurename{Fig.}
\section{Global convergence}
\label{global}
Since Algorithm~\ref{GPA}, GPA, is a special case of the nonmonotone
gradient projection algorithm studied in \cite{hz05a},
our previously established global convergence result, stated below, holds.
\smallskip
\begin{theorem}\label{GPA-theorem}
Let $\C{L}$ be the level set defined by
\begin{equation}\label{level}
\C{L} = \{ \m{x} \in \Omega : f (\m{x}) \le f(\m{x}_0) \} .
\end{equation}
Assume the following conditions hold:
\begin{itemize}
\item[{\rm A1.}]
$f$ is bounded from below on $\C{L}$ and
$d_{\max} = {\sup}_k \|\m{d}_k\| < \infty$.
\item[{\rm A2.}]
If $\bar{\C{L}}$ is the collection of $\m{x}\in \Omega$ whose
distance to $\C{L}$ is at most $d_{\max}$,
then $\nabla f$ is Lipschitz continuous on $\bar{\C{L}}$.
\end{itemize}
Then GPA with $\epsilon = 0$ either terminates in a finite number of
iterations at a stationary point, or we have
\[
\liminf\limits_{k \to\infty} E(\m{x}_k) =0.
\]
\end{theorem}
\smallskip

The global convergence of PASA essentially follows from the global convergence
of GPA and the requirement F3 for the linearly constrained optimizer.
\begin{theorem}\label{global_theorem}
If the assumptions of Theorem~$\ref{GPA-theorem}$ hold and the linearly
constrained optimizer satisfies {\rm F1--F3}, then
PASA with $\epsilon = 0$ either terminates in a finite number of
iterations at a stationary point, or we have
\begin{equation}\label{lim_to_zero}
\liminf\limits_{k \to\infty} E(\m{x}_k)=0.
\end{equation}
\end{theorem}
\smallskip
\begin{proof}
If only phase one is performed for $k$ sufficiently large, then
(\ref{lim_to_zero}) follows from Theorem~\ref{GPA-theorem}.
If only phase two is performed for $k$ sufficiently large, then
$e(\m{x}_k) \ge \theta E(\m{x}_k)$ for $k$ sufficiently large.
Since $\theta$ is only changed in phase one, we can treat
$\theta$ as a fixed positive scalar for $k$ sufficiently large.
By F2, the active sets approach a fixed limit for $k$ sufficiently large.
By F3 and the inequality $e(\m{x}_k) \ge \theta E(\m{x}_k)$,
(\ref{lim_to_zero}) holds.
Finally, suppose that there are an infinite number of branches from
phase two to phase one.
If (\ref{lim_to_zero}) does not hold, then there exists $\tau > 0$
such that $E(\m{x}_k) = \|\m{d}^1(\m{x}_k)\| \ge \tau$ for all $k$.
By property P6 of \cite{hz05a} and the definition
$\m{d}_k = \m{d}^\alpha(\m{x}_k)$, we have
\begin{equation}\label{g1}
\nabla f (\m{x}_k) \m{d}_k = \m{g}_k\tr \m{d}^\alpha (\m{x}_k) \le
-\|\m{d}^\alpha(\m{x}_k)\|^2/\alpha = -\| \m{d}_k\|^2/\alpha ,
\end{equation}
which implies that
\begin{equation}\label{g4}
\frac{|\m{g}_k\tr \m{d}_k|}{\|\m{d}_k\|^2} \ge \frac{1}{\alpha}.
\end{equation}
By P4 and P5 of \cite{hz05a} and the
lower bound $\|\m{d}^1(\m{x}_k)\| \ge \tau$, we have
\begin{equation}\label{g2}
\|\m{d}_k\| =
\|\m{d}^\alpha(\m{x}_k)\| \ge \min \{\alpha , 1\} \|\m{d}^1(\m{x}_k)\| \ge
\min \{\alpha , 1\} \tau.
\end{equation}
For the Armijo line search in GPA, it follows from \cite[Lem.~2.1]{zhh04} that
\[
s_k \ge \min \left\{ 1, \;
\left( \frac{2\eta(1-\delta)}{\kappa} \right)
\frac{|\m{g}_k \tr \m{d}_k|}{\|\m{d}_k\|^2}
\right\}
\]
for all iterations in GPA,
where $\kappa$ is the Lipschitz constant for $\nabla f$.
Combine this with (\ref{g4}) to obtain
\begin{equation}\label{lower}
s_k \ge \min \left\{ 1, \;
\left( \frac{2\eta(1-\delta)}{\kappa \alpha} \right) \right\} .
\end{equation}
By the line search condition in step~2 of GPA,
it follows from (\ref{g1}), (\ref{g2}), and (\ref{lower})
that there exists $c > 0$ such that
\begin{equation}\label{g3}
f(\m{x}_{k+1}) \le f(\m{x}_k) - c
\end{equation}
for all iterations in GPA with $k$ sufficiently large.
Since the objective function decreases monotonically in phase two,
and since there are an infinite number of iterations in GPA,
(\ref{g3}) contradicts the assumption A1 that $f$ is bounded from below.
\end{proof}
\section{Properties of projections}
\label{stationarity}
The proof that only phase two of PASA is executed asymptotically relies
on some properties for the solution of the projection problem (\ref{y-def})
that are established in this section.
Since the projection onto a convex set is a nonexpansive operator, we have
\begin{eqnarray}
\|\m{y}(\m{x}_1, \alpha) - \m{y}(\m{x}_2, \alpha)\| &=&
\|\C{P}_\Omega (\m{x}_1 - \alpha \m{g}(\m{x}_1)) -
\C{P}_\Omega (\m{x}_2 - \alpha \m{g}(\m{x}_2)) \| \nonumber \\
&\le& \|(\m{x}_1 - \alpha \m{g}(\m{x}_1)) -
(\m{x}_2 - \alpha \m{g}(\m{x}_2)) \| \nonumber \\
&\le& (1 + \alpha \kappa) \|\m{x}_1 - \m{x}_2\|, \label{lip}
\end{eqnarray}
where $\kappa$ is a Lipschitz constant for $\m{g}$.
Since $\m{d}^\alpha (\m{x}^*) = \m{0}$ for all $\alpha > 0$ when
$\m{x}^*$ is a stationary point, it follows that
$\m{y}(\m{x}^*, \alpha) = \m{x}^*$ for all $\alpha > 0$.
In the special case where $\m{x}_2 = \m{x}^*$, (\ref{lip}) yields
\begin{equation}\label{lip*}
\|\m{y}(\m{x}, \alpha) - \m{x}^*\| =
\|\m{y}(\m{x}, \alpha) - \m{y}(\m{x}^*, \alpha)\| \le
(1 + \alpha \kappa) \|\m{x} - \m{x}^*\| .
\end{equation}
Similar to (\ref{lip}), but with $\m{y}$ replaced by $\m{d}^\alpha$, we have
\begin{equation}\label{lipd}
\|\m{d}^\alpha (\m{x}_1) - \m{d}^\alpha(\m{x}_2)\| \le
(2 + \alpha \kappa) \|\m{x}_1 - \m{x}_2\|.
\end{equation}

Next, let us develop some properties for the
multipliers associated with the constraint in the (\ref{y-def}).
The first-order optimality conditions associated with (\ref{y-def}) can be
expressed as follows:
At any solution $\m{y} = \m{y}(\m{x}, \alpha)$ of (\ref{y-def}), there exists
a multiplier $\g{\lambda} \in \mathbb{R}^m$ such that
\begin{equation}\label{first-order}
\m{y} - \m{x} + \alpha \m{g}(\m{x}) + \m{A}\tr\g{\lambda}
= \m{0}, \quad \g{\lambda} \ge \m{0}, \quad
\lambda_i = 0 \mbox{ if } i \in \C{F}(\m{y}).
\end{equation}
Let $\Lambda (\m{x}, \alpha)$ denote the set of
multipliers $\g{\lambda}$ satisfying (\ref{first-order}) at the solution
$\m{y} = \m{y}(\m{x}, \alpha)$ of (\ref{y-def}).
If $\m{x}^*$ is a stationary point for (\ref{P}) and $\alpha > 0$,
then $\m{y}(\m{x}^*, \alpha) = \m{x}^*$, and the first equation in
(\ref{first-order}) reduces to
\[
\m{g}(\m{x}^*) + \m{A}\tr(\g{\lambda}/\alpha) = \m{0},
\]
which is the gradient of the Lagrangian for (\ref{P}), but with the
multiplier scaled by $\alpha$.
Since $\C{F}(\m{y}(\m{x}^*,\alpha)) = \C{F}(\m{x}^*)$,
$\g{\lambda}/\alpha$ is a multiplier for the constraint in (\ref{P}).
Thus if $\m{x}^*$ is a stationary point for (\ref{P}) and
$\Lambda(\m{x}^*)$ is the set of Lagrange multipliers associated with
the constraint, we have
\begin{equation}\label{scaling}
\Lambda(\m{x}^*, \alpha) = \alpha \Lambda (\m{x}^*).
\end{equation}
By (\ref{lip*}) $\m{y}(\m{x}, \alpha)$ approaches $\m{x}^*$ as $\m{x}$
approaches $\m{x}^*$.
Consequently, the indices $\C{F}(\m{x}^*)$ free at $\m{x}^*$ are
free at $\m{y}(\m{x}, \alpha)$ when $\m{x}$ is sufficiently close to $\m{x}^*$.
The multipliers associated with (\ref{y-def}) have the following
stability property.
\begin{proposition}\label{SetClose}
Suppose $\m{x}^*$ is a stationary point for $(\ref{P})$ and for
some $r > 0$, $\m{g}$ is Lipschitz continuous in $\C{B}_r (\m{x}^*)$
with Lipschitz constant $\kappa$.
If $\alpha \ge 0$ and
$\m{x} \in \C{B}_r (\m{x}^*)$ is close enough to $\m{x}^*$ that
$\C{F}(\m{x}^*) \subset \C{F}(\m{y}(\m{x}, \alpha))$, then
\[
\mbox{\rm dist}\{ \g{\lambda}, \Lambda (\m{x}^*, \alpha)\} \le
2c (1 + \kappa \alpha) \|\m{x} - \m{x}^*\|
\]
for all $\g{\lambda} \in \Lambda(\m{x}, \alpha)$,
where $c$ is independent of $\m{x}$ and depends only on $\m{A}$.
\end{proposition}

\begin{proof}
This is essentially a consequence of the upper Lipschitzian properties of
polyhedral multifunctions as established in \cite[Prop.~1]{Robinson81} or
\cite[Cor.~4.2]{Robinson82}.
Here is a short proof based on Hoffman's stability result \cite{Hoffman52} 
for a perturbed linear system of inequalities.
Since $\C{F}(\m{x}^*) \subset \C{F}(\m{y}(\m{x}, \alpha))$,
it follows that any $\g{\lambda} \in \Lambda (\m{x}, \alpha)$ is feasible
in the system
\[
\m{p} + \m{A}\tr \g{\lambda} = \m{0}, \quad
\g{\lambda} \ge \m{0}, \quad \lambda_i = 0 \mbox{ if } i \in \C{F}(\m{x}^*),
\]
with $\m{p} = \m{p}_1 := \m{y}(\m{x}, \alpha) - \m{x} + \alpha \m{g}(\m{x})$.
Since $\m{x}^*$ is a stationary point for (\ref{P}),
the elements of $\Lambda (\m{x}^*, \alpha)$ are feasible in the same
system but with
\[
\m{p} = \m{p}_2 := \m{x}^* - \m{x}^* + \alpha \m{g}(\m{x}^*).
\]
Hence, by Hoffman's result \cite{Hoffman52}, there exists a constant $c$,
independent of $\m{p}_1$ and $\m{p}_2$ and depending only on $\m{A}$,
such that
\[
\mbox{\rm dist}\{ \g{\lambda}, \Lambda (\m{x}^*, \alpha)\} \le
c \|\m{p}_1 - \m{p}_2\|.
\]
We use (\ref{lip*}) to obtain
\[
\|\m{p}_1 - \m{p}_2\| =
\|(\m{y}(\m{x}, \alpha) - \m{x}^*) + (\m{x}^* - \m{x}) + \alpha
(\m{g}(\m{x}) - \m{g}(\m{x}^*))\| \le
2 (1 + \alpha \kappa)\|\m{x} - \m{x}^*\|,
\]
which completes the proof.
\end{proof}
\smallskip

In the following proposition, we study the projection of a step
$\m{x}_k - \alpha \m{g}(\m{x}_k)$ onto the subset $\Omega_k$ of $\Omega$
that also satisfies the active constraints at $\m{x}_k$.
We show that the step can be replaced by
$\m{x}_k - \alpha \m{g}^\C{A}(\m{x}_k)$ without effecting the projection.
\smallskip
\begin{proposition}
\label{P=prop}
For all $\alpha \ge 0$, we have
\begin{equation}\label{P=}
\C{P}_{\Omega_k} (\m{x}_k - \alpha \m{g}(\m{x}_k)) =
\C{P}_{\Omega_k} (\m{x}_k - \alpha \m{g}^\C{A}(\m{x}_k)),
\end{equation}
where
\begin{equation}\label{omegak}
\Omega_k = \{ \m{x} \in \Omega : (\m{Ax} - \m{b})_i = 0 \mbox{ for all }
i \in \C{A}(\m{x}_k) \}.
\end{equation}
\end{proposition}
\smallskip

\begin{proof}
Let $\m{p}$ be defined by
\begin{eqnarray}
\m{p} &=& \C{P}_{\Omega_k}
(\m{x}_k - \alpha \m{g}(\m{x}_k)) - \m{x}_k
\label{1}\\
&=& \arg \; \min \{
\| \m{x}_k - \alpha \m{g}(\m{x}_k) - \m{y}\| : \m{y} \in \Omega_k \} - \m{x}_k .
\nonumber
\end{eqnarray}
With the change of variables $\m{z} = \m{y} - \m{x}_k$, we can write
\begin{equation} \label{nd1}
\m{p} = \arg \; \min \{
\| \m{z} + \alpha \m{g}(\m{x}_k)\| : \m{z} \in \Omega_k - \m{x}_k \}.
\end{equation}
%
Since $ \m{g}^\C{A} (\m{x}_k)$ is the orthogonal projection of
$\m{g}(\m{x}_k)$ onto the null space $\C{N}(\m{A}_{\C{I}})$,
where $\C{I} = \C{A}(\m{x}_k)$,
the difference $\m{g}^\C{A} (\m{x}_k) - \m{g}(\m{x}_k)$ is orthogonal to
$\C{N}(\m{A}_{\C{I}})$.
Since $\Omega_k - \m{x}_k \subset \C{N}(\m{A}_{\C{I}})$, it follows from
Pythagoras that for any $\m{z} \in \Omega_k - \m{x}_k$, we have
\[
\| \m{z} + \alpha \m{g}(\m{x}_k) \|^2  =
\| \m{z} + \alpha \m{g}^\C{A}(\m{x}_k)  \|^2 +
\alpha^2 \|
\m{g}(\m{x}_k) - \m{g}^\C{A}(\m{x}_k) \|^2.
\]
Since $\m{z}$ does not appear in the last term, minimizing
$\| \m{z} + \alpha \m{g}(\m{x}_k) \|^2$ over $\m{z} \in \Omega_k - \m{x}_k$
is equivalent to minimizing $\| \m{z} + \m{g}^\C{A}(\m{x}_k)\|^2$ over
$\m{z} \in \Omega_k - \m{x}_k$.
By (\ref{nd1}), we obtain
\begin{equation}\label{nd2}
\m{p} = \arg \; \min \{
\| \m{z} + \alpha \m{g}^\C{A}(\m{x}_k)\| : \m{z} \in \Omega_k - \m{x}_k \}.
\end{equation}
Changing variables from $\m{z}$ back to $\m{y}$ gives
\begin{eqnarray*}
\m{p} &=& \arg \; \min \{
\| \m{x}_k - \alpha \m{g}^\C{A}(\m{x}_k) - \m{y}\| :
\m{y} \in \Omega_k \} - \m{x}_k  \\
&=& \C{P}_{\Omega_k}(\m{x}_k - \alpha \m{g}^\C{A}(\m{x}_k)) - \m{x}_k .
\end{eqnarray*}
Comparing this to (\ref{1}) gives (\ref{P=}).
\end{proof}
\smallskip

\section{Nondegenerate problems}
\label{nondegenerate}
In this section, we focus on the case where the iterates of PASA converge
to a nondegenerate stationary point; that is, a stationary point $\m{x}^*$
for which there exists a scalar $\pi > 0$ such that
$\lambda_i > \pi$ for all $i \in \C{A}(\m{x}^*)$ and
$\g{\lambda} \in \Lambda(\m{x}^*)$.
\smallskip
\begin{theorem}
\label{nondegenerate-theorem}
If PASA with $\epsilon = 0$ generates an infinite sequence of iterates
that converge to a nondegenerate stationary point $\m{x}^*$,
then within a finite number of iterations, only phase two is executed.
\end{theorem}
\smallskip
\begin{proof}
%
By (\ref{lip*}) and Proposition~\ref{SetClose},
$\m{y}(\m{x}, 1)$ is close to $\m{x}^*$ and $\Lambda(\m{x}, 1)$
is close to $\Lambda(\m{x}^*)$ when $\m{x}$ is close to $\m{x}^*$.
It follows that for $r$ sufficiently small, we have
\begin{equation}\label{lambda>0}
\lambda_i > 0 \mbox{ for all }
i \in \C{A}(\m{x}^*), \; \g{\lambda} \in \Lambda(\m{x}, 1),
\mbox{ and } \m{x} \in \C{B}_r (\m{x}^*).
\end{equation}
Since $(\m{Ax}^*-\m{b})_i < 0$ for all $i \in \C{F}(\m{x}^*)$, it also
follows from (\ref{lip*}) that we can take $r$ smaller, if necessary,
to ensure that for all $i \in \C{F}(\m{x}^*)$ and
$\m{x} \in \C{B}_r (\m{x}^*)$, we have
\begin{equation}\label{Ax<b}
(\m{Ay}(\m{x}, 1) - \m{b})_i   < 0 \quad \mbox{and} \quad
(\m{Ax}-\m{b})_i < 0.
\end{equation}
By the last condition in (\ref{Ax<b}), we have
\begin{equation}\label{601}
\C{A}(\m{x}) \subset \C{A}(\m{x}^*) \quad \mbox{for all } \m{x} \in
\C{B}_r(\m{x}^*).
\end{equation}
By (\ref{lambda>0}) and (\ref{Ax<b}),
\begin{equation}\label{600}
\C{A}(\m{y}(\m{x}, 1)) = \C{A}(\m{x}^*) \quad
\mbox{for all } \m{x} \in \C{B}_r(\m{x}^*).
\end{equation}
That is, if $\m{x} \in \C{B}_r(\m{x}^*)$ and
$i \in \C{A}(\m{x}^*)$, then by (\ref{lambda>0}) and complementary slackness,
$i$ lies in $\C{A}(\m{y}(\m{x}, 1))$, which implies that
$\m{A}(\m{x}^*) \subset \C{A}(\m{y}(\m{x},1))$.
Conversely, if $i \in \C{F}(\m{x}^*) = \m{A}(\m{x}^*)^c$, then by (\ref{Ax<b}),
$i$ lies in $\C{F}(\m{y}(\m{x}, 1)) = \C{A}(\m{y}(\m{x},1))^c$.
Hence, (\ref{600}) holds.

Choose $K$ large enough that $\m{x}_k \in \C{B}_r(\m{x}^*)$ for all
$k \ge K$.
Since $\C{A}(\m{x}_k) \subset \C{A}(\m{x}^*) = \C{A}(\m{y}(\m{x}_k, 1))$
for all $k \ge K$ by (\ref{601}) and (\ref{600}), it follows that
\begin{equation}\label{602}
\C{P}_\Omega (\m{x}_k - \m{g}(\m{x}_k)) =
\m{y}(\m{x}_k, 1) \in \Omega_k \subset \Omega,
\end{equation}
where $\Omega_k$ is defined in (\ref{omegak}).
The inclusion (\ref{602}) along with Proposition~\ref{P=prop} yield
\[
\C{P}_\Omega (\m{x}_k - \m{g}(\m{x}_k)) =
\C{P}_{\Omega_k} (\m{x}_k - \m{g}(\m{x}_k)) =
\C{P}_{\Omega_k} (\m{x}_k - \m{g}^\C{A}(\m{x}_k)).
\]
We subtract $\m{x}_k$ from both sides and refer to the
definition (\ref{dalpha}) of $\m{d}^\alpha$ to obtain
\begin{eqnarray}
\m{d}^1(\m{x}_k) &=&
\C{P}_{\Omega_k} (\m{x}_k - \m{g}^\C{A}(\m{x}_k)) - \m{x}_k \nonumber \\
&=& \arg \; \min \left\{
\| \m{x}_k - \m{g}^\C{A}(\m{x}_k) - \m{y}\|^2 :
\m{y} \in \Omega_k \right\} - \m{x}_k \nonumber \\
&=& \arg \; \min \left\{
\| \m{z} + \m{g}^\C{A}(\m{x}_k)\|^2 :
\m{z} \in \Omega_k - \m{x}_k \right\} . \label{10}
\end{eqnarray}
Recall that at a local minimizer $\bar{\m{x}}$ of a smooth function
$F$ over the convex set $\Omega_k$, the variational inequality
$\nabla F(\bar{\m{x}}) (\m{x} - \bar{\m{x}}) \ge 0$ holds for all
$\m{x} \in \Omega_k$.
We identify $F$ with the objective in (\ref{10}),
$\bar{\m{x}}$ with $\m{d}^1 (\m{x}_k)$,
and $\m{x}$ with the point $\m{0} \in \Omega_k - \m{x}_k$
to obtain the inequality
\[
\m{d}^1(\m{x}_k)  \tr ( \m{g}^\C{A} (\m{x}_k)  +  \m{d}^1(\m{x}_k) ) \le 0.
\]
Hence,
\begin{eqnarray*}
\|\m{g}^\C{A} (\m{x}_k)  \|^2 &=&
\| \m{g}^\C{A} (\m{x}_k) + \m{d}^1(\m{x}_k)  \|^2
- 2  \m{d}^1(\m{x}_k)  \tr ( \m{g}^\C{A} (\m{x}_k) + \m{d}^1(\m{x}_k) )
+ \| \m{d}^1(\m{x}_k) \|^2 \\
&\ge&  \|\m{d}^1(\m{x}_k)  \|^2.
\end{eqnarray*}
By definition, the left side of this inequality is $e(\m{x}_k)^2$, while
the right side is $E(\m{x}_k)^2$.
Consequently, $E(\m{x}_k) \le e(\m{x}_k)$ when $k \ge K$.
Since $\theta \in (0,1)$, it follows that phase one immediately
branches to phase two, while phase two cannot branch to phase one.
This completes the proof.
\end{proof}
\section{Degenerate problems}
\label{degenerate}
We now focus on a degenerate stationary point $\m{x}^*$ where
there exists $i \in \C{A}(\m{x}^*)$ and $\g{\lambda} \in \Lambda (\m{x}^*)$
such that $\lambda_i = 0$.
We wish to establish a result analogous to Theorem~\ref{nondegenerate-theorem}.
To compensate for the degeneracy, it is assumed that the active constraint
gradients at $\m{x}^*$ are linearly independent;
that is, the rows of $\m{A}$ corresponding to indices $i \in \C{A}(\m{x}^*)$
are linearly independent,
which implies that $\Lambda (\m{x}^*, \alpha)$ is a singleton.
Under this assumption, Proposition~\ref{SetClose} yields the
following Lipschitz property.
\begin{corollary}\label{SetLip}
Suppose $\m{x}^*$ is a stationary point for $(\ref{P})$ and the active
constraint gradients are linearly independent at $\m{x}^*$.
If for some $r > 0$,
$\m{g}$ is Lipschitz continuous in $\C{B}_r (\m{x}^*)$
with Lipschitz constant $\kappa$ and
$\m{x} \in \C{B}_r (\m{x}^*)$ is close enough to $\m{x}^*$ that
$\C{F}(\m{x}^*) \subset \C{F}(\m{y}(\m{x}, \alpha))$ for
some $\alpha \ge 0$, then $\Lambda(\m{x}, \alpha)$ is a singleton and
\[
\| \Lambda(\m{x}, \alpha) - \Lambda (\m{x}^*, \alpha) \|
\le 2c (1 + \kappa \alpha) \|\m{x} - \m{x}^*\|,
\]
where $c$ is independent of $\m{x}$ and depends only on $\m{A}$.
\end{corollary}
\smallskip
\begin{proof}
Since $\C{F}(\m{x}^*) \subset \C{F}(\m{y}(\m{x}, \alpha))$, it follows that
$\C{A}(\m{x}^*) \supset \C{A}(\m{y}(\m{x}, \alpha))$.
Hence, the active constraint gradients are linearly independent at
$\m{x}^*$ and at $\m{y}(\m{x}, \alpha)$.
This implies that both
$\Lambda(\m{x}, \alpha)$ and $\Lambda (\m{x}^*, \alpha)$ are singletons, and
Corollary~\ref{SetLip} follows from Proposition~\ref{SetClose}.
\end{proof}

To treat degenerate problems, the convergence theory involves one more
requirement for the linearly constrained optimizer:
\smallskip
\begin{itemize}
\item[F4.]
When branching from phase one to phase two, the first iteration in phase two
is given by an Armijo line search of the following form: Choose $j \ge 0$
as small as possible such that
\begin{eqnarray}
f (\m{x}_{k+1}) &\le& f(\m{x}_k) +
\delta \nabla f(\m{x}_k) (\m{x}_{k+1} - \m{x}_k)
\mbox{ where} \label{p-armijo} \\
\m{x}_{k+1} &=& \C{P}_{\Omega_k} (\m{x}_k - s_k \m{g}^\C{A}(\m{x}_k)),
\quad s_k = \alpha \eta^j,
\nonumber
\end{eqnarray}
with $\Omega_k$ defined in (\ref{omegak}),
$\delta\in (0, 1)$, $\eta \in (0, 1)$, and $\alpha \in (0, \infty)$
(as in the Armijo line search of GPA).
\end{itemize}
\smallskip
As $j$ increases, $\eta^j$ tends to zero and $\m{x}_{k+1}$ approaches
$\m{x}_k$.
Hence, for $j$ sufficiently large,
$i \in \C{F}(\m{x}_k - \alpha \eta^j \m{g}^\C{A}(\m{x}_k))$ if
$i \in \C{F}(\m{x}_k)$.
Since $(\m{A} \m{g}^\C{A}(\m{x}_k))_i = 0$ if $i \in \C{A}(\m{x}_k)$,
it follows that $\m{x}_k - \alpha \eta^j \m{g}^\C{A}(\m{x}_k) \in \Omega_k$
for $j$ sufficiently large, which implies that
$\m{x}_{k+1} =$ $\m{x}_k - \alpha \eta^j \m{g}^\C{A}(\m{x}_k)$;
consequently, for $j$ sufficiently large,
the Armijo line search inequality (\ref{p-armijo}) reduces to the
ordinary Armijo line search condition
\[
f (\m{x}_{k+1}) \le f(\m{x}_k) -
s_k \delta \nabla f(\m{x}_k) \m{g}^\C{A}(\m{x}_k),
\]
which holds for $s_k$ sufficiently small.
The basic difference between the Armijo line search in F4 and the
Armijo line search in GPA is that in F4, the constraints active at
$\m{x}_k$ remain active at $\m{x}_{k+1}$ and F2 holds.
With the additional startup procedure F4 for LCO,
the global convergence result Theorem~\ref{global_theorem} remains
applicable since conditions F1 and F2 are satisfied by the initial
iteration in phase two.

Let $\m{x}^*$ be a stationary point where the active constraint
gradients are linearly independent.
For any given $\m{x} \in \mathbb{R}^n$, we define
\begin{equation}\label{fstar}
\bar{\m{x}} = \arg \; \min_{\m{y}} \{\|\m{x} - \m{y}\| :
(\m{Ay} - \m{b})_i = 0 \mbox{ for all }
i \in \C{A}_+(\m{x}^*) \cup \C{A}(\m{x}) \} ,
\end{equation}
where $\C{A}_+(\m{x}^*) = \{i \in \C{A}(\m{x}^*): \Lambda_i (\m{x}^*) > 0\}$.
If $\m{x}$ is close enough to $\m{x}^*$ that
$\m{A}(\m{x}) \subset \m{A}(\m{x}^*)$, then the feasible set in (\ref{fstar})
is nonempty since $\m{x}^*$ satisfies the constraints;
hence, the projection in (\ref{fstar}) is nonempty when $\m{x}$ is sufficiently
close to $\m{x}^*$.
\smallskip
\begin{lemma}
\label{ratio}
Suppose $\m{x}^*$ is a stationary point where the active constraint gradients
are linearly independent and $f$ is Lipschitz continuously differentiable
in a neighborhood of $\m{x}^*$.
If PASA with $\epsilon=0$ generates an infinite sequence of iterates $\m{x}_k$
converging to $\m{x}^*$, then there exists $c \in \mathbb{R}$ such that
\begin{equation}\label{Aidentify}
\| \m{x}_k - \bar{\m{x}}_k \| \le c \|\m{x}_k - \m{x}^*\|^2
\end{equation}
for all $k$ sufficiently large.
\end{lemma}
\smallskip
\begin{proof}
Choose $r$ in accordance with Corollary~\ref{SetLip}.
Similar to what is done in the proof of Theorem~\ref{nondegenerate-theorem},
choose $r > 0$ smaller if necessary to ensure that
for all $\m{x} \in \C{B}_r (\m{x}^*)$, we have
\begin{equation}\label{700}
\Lambda_i (\m{x}, \alpha) > 0 \mbox{ for all }
i \in \C{A}_+(\m{x}^*),
\end{equation}
and
\begin{equation}\label{701}
(\m{Ay}(\m{x}, \alpha) - \m{b})_j < 0, \quad (\m{Ax}-\m{b})_j < 0,
\mbox{ for all } j \in \C{F}(\m{x}^*).
\end{equation}
Choose $K$ large enough that $\m{x}_k \in \C{B}_r(\m{x}^*)$ for all $k \ge K$
and suppose that $\m{x}_k$ is any PASA iterate with $k \ge K$.
If $i \in \C{A}_+(\m{x}^*) \cap \C{A}(\m{x}_k)$, then by
(\ref{700}), $i \in \C{A}(\m{y}(\m{x}_k, \alpha))$ by complementary slackness.
Hence, $i \in \C{A}(\m{x})$ for all $\m{x}$ on the line segment
connecting $\m{x}_k$ and $\m{y}(\m{x}_k, \alpha)$.
In particular, $i \in \C{A}(\m{x}_{k+1})$ if $\m{x}_{k+1}$ is generated
by GPA in phase one, while
$i \in \C{A}(\m{x}_{k+1})$ by F2 if $\m{x}_{k+1}$ is generated in phase two.
It follows that if
constraint $i \in \C{A}_+(\m{x}^*)$ becomes active at iterate $\m{x}_k$,
then $i \in \C{A} (\m{x}_l)$ for all $l \ge k$.
Let $\C{I}$ be the limit of $\C{A}_+(\m{x}^*) \cap \C{A}(\m{x}_k)$
as $k$ tends to infinity;
choose $K$ larger if necessary to ensure that
$\C{I} \subset \C{A}(\m{x}_k)$ for all $k \ge K$ and suppose that
$\m{x}_k$ is any iterate of PASA with $k \ge K$.
If $\C{I} = \C{A}_+(\m{x}^*)$, then since $\C{I} \subset \C{A}(\m{x}_k)$,
it follows that
$\C{A}_+(\m{x}^*) \cup \C{A}(\m{x}_k) = \C{A}(\m{x}_k)$,
which implies that $\bar{\m{x}}_k = \m{x}_k$.
Thus (\ref{Aidentify}) holds trivially since the left side vanishes.
Let us focus on the nontrivial case where $\C{I}$ is strictly contained in
$\C{A}_+(\m{x}^*)$.
The analysis is partitioned into three cases.

{\bf Case 1.} {\em For $k$ sufficiently large,
$\m{x}_k$ is generated solely by LCO.}
By F3 it follows that for any $\epsilon > 0$, there exists $k \ge K$ such that
$\|\m{g}^\C{A}(\m{x}_k)\| = e (\m{x}_k) \le \epsilon$.
By the first-order optimality conditions for $\m{g}^\C{A}(\m{x}_k)$,
there exists $\g{\mu}_k \in \mathbb{R}^m$,
with $\mu_{ki} = 0$ for all $i \in \C{F}(\m{x}_k)$,
such that
\begin{equation}\label{702}
\| \m{g} (\m{x}_k) + \m{A}\tr \g{\mu}_k \| \le \epsilon.
\end{equation}
The multiplier $\g{\mu}_k$ is unique by
the independence of the active constraint gradients and the fact that
$\C{A}(\m{x}_k) \subset \C{A}(\m{x}^*)$ by the last condition in (\ref{701}).
Similarly, at $\m{x}^*$ we have
$\m{g} (\m{x}^*) + \m{A}\tr \g{\lambda}^* = \m{0}$,
where $\g{\lambda}^* = \Lambda(\m{x}^*)$.
Combine this with (\ref{702}) to obtain
\begin{equation}\label{703}
\|\m{A}\tr (\g{\mu}_k - \g{\lambda}^*)\| \le \epsilon +
\|\m{g}(\m{x}_k) - \m{g}(\m{x}^*)\| \le \epsilon + \kappa \|\m{x}_k - \m{x}^*\|.
\end{equation}
Since $\C{F}(\m{x}^*) \subset \C{F}(\m{x}_k)$ by
the last condition in (\ref{701}), it follows from complementary slackness
that $\mu_{ki} = \lambda_i^* = 0$ for all $i \in \C{F}(\m{x}^*)$.
Since the columns of $\m{A}\tr$ corresponding to indices in $\C{A}(\m{x}^*)$
are linearly independent, there exists a constant $c$ such that
\begin{equation}\label{704}
\|\g{\mu}_k - \g{\lambda}^*\| \le c \|\m{A}\tr (\g{\mu}_k - \g{\lambda}^*)\|.
\end{equation}
Hence, for $\epsilon$ sufficiently small and $k$ sufficiently large,
it follows from (\ref{703}) and (\ref{704}) that
$\mu_{ki} > 0$ for all $i \in \C{A}_+(\m{x}^*)$, which contradicts the
assumption that $\C{I}$ is strictly contained in $\C{A}_+(\m{x}^*)$.
Consequently, case 1 cannot occur.

{\bf Case 2.} {\em PASA makes an infinite number of branches from phase one
to phase two and from phase two to phase one.}
Let us consider the first iteration of phase two.
By Proposition~\ref{P=prop} and the definition of $\m{x}_{k+1}$ in F4,
we have
\[
\m{x}_{k+1} = \C{P}_{\Omega_k} (\m{x}_k - s_k \m{g}(\m{x}_k)).
\]
The first-order optimality condition for $\m{x}_{k+1}$ is that there
exists $\g{\mu}_k \in \mathbb{R}^m$ such that
\[
\m{x}_{k+1} - \m{x}_k + s_k \m{g}(\m{x}_k) + \m{A}\tr \g{\mu}_k
= \m{0},
\]
where $\mu_{ki} = 0$ for all $i \in \C{F}(\m{x}_{k+1}) \supset \C{F}(\m{x}^*)$.
Subtracting from this the identity
\[
s_k \m{g}(\m{x}^*) + \m{A}\tr (s_k \g{\lambda}^*) = \m{0}
\]
yields
\begin{equation}\label{800}
\m{A}\tr (\g{\mu}_k - s_k \g{\lambda}^*) =
\m{x}_k - \m{x}_{k+1} + s_k (\m{g}(\m{x}^*) - \m{g}(\m{x}_k)) .
\end{equation}
By the Lipschitz continuity of $\m{g}$,
the bound $s_k \le \alpha$ in F4,
and the assumption that the $\m{x}_k$
converge to $\m{x}^*$, the right side of (\ref{800})
tends to zero as $k$ tends to infinity.
Exploiting the independence of the active constraint gradients and
the identity $\mu_{ki} = \lambda_i^* = 0$ for all $i \in \C{F}(\m{x}^*)$,
we deduce from (\ref{704}) and (\ref{800}) that
$\|\g{\mu}_k - s_k \g{\lambda}^*\|$ tends to 0 as $k$ tends to
infinity.
It follows that for each $i$, $\mu_{ki} - s_k \lambda_i^*$
tends to zero.
If $s_k$ is uniformly bounded away from 0,
then $\mu_{ki} > 0$ when $i \in \C{A}_+(\m{x}^*)$.
By complementary slackness, $\C{I} = \C{A}_+(\m{x}^*)$, which
would contradict the assumption that $\C{I}$ is strictly
contained in $\C{A}_+(\m{x}^*)$.
Consequently, case 2 could not occur.

We will now establish a positive lower bound for $s_k$ in F4 of phase two.
If the Armijo stepsize terminates at $j = 0$, then $s_k = \alpha > 0$,
and we are done.
Next, suppose the stepsize terminates at $j \ge 1$.
Since $j$ is as small as possible, it follows from
Proposition~\ref{P=prop} and F4 that
\begin{equation}\label{900}
f(\m{x}_k + \m{d}_k) - f(\m{x}_k) > \delta \m{g}_k\tr\m{d}_k,
\end{equation}
where
\[
\m{d}_k =\C{P}_{\Omega_k} (\m{x}_k - \beta \m{g}_k) - \m{x}_k, \quad
\beta := s_k/\eta \le \alpha.
\]
The inequality $\beta \le \alpha$ holds since $j \ge 1$.
Since $\m{x}^* \in \Omega_k \subset \Omega$ by the second
condition in (\ref{701}), we have
\[
\C{P}_{\Omega_k} (\m{x}^*-\beta\m{g}(\m{x}^*)) = \m{x}^*.
\]
Since the projection onto a convex set is a nonexpansive operator,
we obtain
\[
\|(\m{x}_k + \m{d}_k) - \m{x}^*\| =
\|\C{P}_{\Omega_k}(\m{x}_k - \beta \m{g}(\m{x}_k)) -
\C{P}_{\Omega_k}(\m{x}^* - \beta \m{g}(\m{x}^*)) \| \le
(1 + \alpha \kappa) \|\m{x}_k - \m{x}^*\|.
\]
The right side of this inequality tends to zero as $k$ tends to infinity.
Choose $k$ large enough that
$\m{x}_k + \m{d}_k$ is within the ball centered at $\m{x}^*$ where
$f$ is Lipschitz continuously differentiable.

Let us expand $f$ in a Taylor series around $\m{x}_k$ to obtain
\begin{eqnarray}
f(\m{x}_k + \m{d}_k) - f(\m{x}_k)  &=& \int_{0}^1
f'(\m{x}_k + t\m{d}_k) dt \nonumber \\
&=& \m{g}_k\tr\m{d}_k + \int_{0}^1
(\nabla f(\m{x}_k + t\m{d}_k) - \nabla f (\m{x}_k))\m{d}_k dt \nonumber \\
&\le& \m{g}_k\tr\m{d}_k +
0.5 \kappa\|\m{d}_k\|^2 . \label{904}
\end{eqnarray}
This inequality combined with (\ref{900}) yields
\begin{equation}\label{901}
(1-\delta)\m{g}_k\tr\m{d}_k + 0.5 \kappa \|\m{d}_k\|^2 > 0.
\end{equation}
As in (\ref{g1}), but with $\Omega$ replaced by $\Omega_k$, we have
\begin{equation}\label{902}
\m{g}_k\tr\m{d}_k \le - \|\m{d}_k\|^2/\beta.
\end{equation}
Note that $\m{d}_k \ne \m{0}$ due to (\ref{900}).
Combine (\ref{901}) and (\ref{902}) and replace $\beta$ by $s_k/\eta$
to obtain
\begin{equation}\label{903}
s_k > 2(1-\delta)\eta/\kappa.
\end{equation}
Hence,
if $j \ge 1$ in F4, then $s_k$ has the lower bound given in (\ref{903})
for $k$ sufficiently large, while $s_k = \alpha$ if $j = 0$.
This completes the proof of case 2.

{\bf Case 3.} {\em For $k$ sufficiently large, $\m{x}_k$ is generated
solely by GPA.}
The Taylor expansion (\ref{904}) can be written
\begin{equation}
f(\m{x}_k + \m{d}_k) = f(\m{x}_k) +
\delta\m{g}_k\tr\m{d}_k +
(1-\delta)\m{g}_k\tr\m{d}_k +
0.5 \kappa\|\m{d}_k\|^2, \label{100}
\end{equation}
where $\m{d}_k = \C{P}_\Omega (\m{x}_k - \alpha \m{g}(\m{x}_k)) - \m{x}_k$
is as defined in GPA.
If (\ref{Aidentify}) is violated, then for any choice of $c > 0$,
there exists $k\ge K$ such that
\begin{equation}\label{clower}
\| \m{x}_k - \bar{\m{x}}_k \| > c \|\m{x}_k - \m{x}^*\|^2 .
\end{equation}
By taking $c$ sufficiently large, we will show that
\begin{equation}\label{key0}
(1-\delta)\m{g}_k\tr\m{d}_k + 0.5 \kappa\|\m{d}_k\|^2 \le 0.
\end{equation}
In this case, (\ref{100}) implies that
$s_k = 1$ is accepted in GPA and
\[
\m{x}_{k+1} = \C{P}_\Omega (\m{x}_k - \alpha \m{g}(\m{x}_k)).
\]
By Corollary~\ref{SetLip},
$\|\Lambda(\m{x}_k, \alpha) - \Lambda(\m{x}^*,\alpha)\| =$
$\|\Lambda(\m{x}_k, \alpha) - \alpha \g{\lambda}^*\|$
tends to 0 as $k$ tends to infinity.
This implies that $\Lambda_{i}(\m{x}_k, \alpha) > 0$ when $\lambda_i^* > 0$,
which contradicts the assumption that $\C{I}$ is strictly contained in
$\C{A}_+(\m{x}^*)$.
Hence, (\ref{Aidentify}) cannot be violated.

To establish (\ref{key0}), first observe that
\begin{eqnarray}
\|\m{d}_k\| &=& \|\m{y}(\m{x}_k, \alpha) - \m{x}_k\| \le
\|\m{y}(\m{x}_k, \alpha) - \m{x}^*\| +
\|\m{x}^* - \m{x}_k\| \nonumber \\
&\le & (2+\alpha\kappa)\|\m{x}_k - \m{x}^*\|
\label{key1}
\end{eqnarray}
by (\ref{lip*}).
By the first-order optimality condition (\ref{first-order})
for $\m{y}(\m{x}_k, \alpha)$, it follows that
\[
\m{d}_k = -(\alpha \m{g}(\m{x}_k) + \m{A}\tr \Lambda (\m{x}_k, \alpha)).
\]
The dot product of this equation with $\m{d}_k$ gives
\begin{equation}\label{key2}
(\alpha \m{g}_k + \m{A}\tr \Lambda(\m{x}_k, \alpha))\tr \m{d}_k =
-\|\m{d}_k\|^2 \le 0.
\end{equation}
Again, by the definition of $\m{d}_k$ and by complementary slackness,
we have
\begin{equation}\label{key3}
\Lambda(\m{x}_k, \alpha)\tr\m{Ad}_k =
\Lambda(\m{x}_k, \alpha)\tr\m{A}(\m{y}(\m{x}_k,\alpha) - \m{x}_k) =
\Lambda(\m{x}_k, \alpha)\tr(\m{b} - \m{A}\m{x}_k) .
\end{equation}
By Corollary~\ref{SetLip}, it follows that for $K$ sufficiently large and
for any $k \ge K$,
\[
\Lambda_i(\m{x}_k, \alpha) \ge 0.5 \Lambda_i(\m{x}^*, \alpha)
\mbox{ for all } i \in \C{A}_+(\m{x}^*) .
\]
Hence, for any $i \in \C{A}_+(\m{x}^*)$ and $k \ge K$, (\ref{key3}) gives
\begin{equation}\label{key4}
\Lambda(\m{x}_k, \alpha)\tr\m{A}\m{d}_k \ge
0.5\Lambda_i(\m{x}^*, \alpha)(\m{b} - \m{A}\m{x}_k)_i
= 0.5 \alpha \Lambda_i (\m{x}^*)(\m{b} - \m{A}\m{x}_k)_i
\end{equation}
since $\Lambda(\m{x}_k, \alpha) \ge \m{0}$, $\m{A}\m{x}_k \le \m{b}$,
and each term in the inner product
$\Lambda(\m{x}_k, \alpha)\tr(\m{b} - \m{A}\m{x}_k)$ is nonnegative.
Combine (\ref{key2})--(\ref{key4}) to obtain
\begin{equation}\label{key5}
\m{g}_k\tr \m{d}_k \le
-0.5 \Lambda_i (\m{x}^*)(\m{b} - \m{A}\m{x}_k)_i
\end{equation}
for any $i \in \C{A}_+(\m{x}^*)$ and $k \ge K$.
The distance $\|\m{x}_k - \bar{\m{x}}_k\|$ between
$\m{x}_k$ and its projection $\bar{\m{x}}_k$ in (\ref{fstar})
is bounded by a constant times the
maximum violation of the constraint $(\m{b} - \m{Ax}_k)_i = 0$
for $i \in \C{A}_+(\m{x}^*) \cup \C{A}(\m{x}_k)$; that is,
there exists a constant $\bar{c}$ such that
\begin{equation}\label{101}
\|\m{x}_k - \bar{\m{x}}_k\| \le \bar{c}
\max \{ (\m{b} - \m{Ax}_k)_i : i \in \C{A}_+(\m{x}^*) \cup \C{A}(\m{x}_k) \}.
\end{equation}
Since $(\m{b} - \m{Ax}_k)_i = 0$ for all $i \in \C{A}(\m{x}_k)$,
it follows that the maximum constraint violation in (\ref{101}) is achieved for
some $i \in \C{A}_+(\m{x}^*)$
(otherwise, $\bar{\m{x}}_k = \m{x}_k$ and (\ref{clower}) is violated).
Consequently, if the index $i \in \C{A}_+(\m{x}^*)$ in (\ref{key5})
is chosen to make $(\m{b} - \m{Ax}_k)_i$ as large as possible, then
\[
\m{g}_k\tr \m{d}_k \le -d\|\m{x}_k - \bar{\m{x}}_k\|, \quad
\mbox{where } d = 0.5 \bar{c} \min \{ \Lambda_i (\m{x}^*) :
i \in \C{A}_+(\m{x}^*) \}.
\]
If (\ref{clower}) holds, then by (\ref{key1}), we have
\[
\m{g}_k\tr \m{d}_k \le -c d \|\m{x}_k - \m{x}^*\|^2 \le
-c d \|\m{d}_k\|^2 / (2 + \alpha \kappa ).
\]
Hence, the expression (\ref{key0}) has the upper bound
\[
(1-\delta)\m{g}_k\tr\m{d}_k + 0.5\kappa\|\m{d}_k\|^2 \le
\left( \frac{cd (\delta - 1)}{2 + \alpha \kappa} +0.5\kappa \right) \|\m{d}_k\|^2.
\]
Since $\delta < 1$, this is nonpositive when $c$ is sufficiently large.
This completes the proof of (\ref{key0}).
\end{proof}

Note that there is a fundamental difference between the
prototype GPA used in this paper and the versions of the
gradient projection algorithm based on a
piecewise projected gradient such as those in \cite{bm88, bm94, bmt90}.
In GPA there is a single projection followed by
a backtrack towards the starting point.
Consequently, we are unable to show that the active constraints
are identified in a finite number of iterations,
unlike the piecewise projection schemes,
where the active constraints can be identified in a finite number
of iterations, but at the expense additional projections when the
stepsize increases.
In Lemma \ref{ratio} we show that even though we do not identify the
active constraints, the violation of the constraints
$(\m{Ax} - \m{b})_i = 0$ for $i \in \C{A}_+(\m{x}^*)$ by iterate
$\m{x}_k$ is on the order of the error in $\m{x}_k$ squared.

When $\m{x}^*$ is fully determined by the active constraints for which the
strict complementarity holds, convergence is achieved
in a finite number of iterations as we now show.
\smallskip
\begin{corollary}
Suppose $\m{x}^*$ is a stationary point where the active constraint gradients
are linearly independent and $f$ is Lipschitz continuously differentiable
in a neighborhood of $\m{x}^*$.
If the PASA iterates $\m{x}_k$ converge to $\m{x}^*$ and
$|\C{A}_+(\m{x}^*)| = n$,
then $\m{x}_k = \m{x}^*$ after a finite number of iterations.
\end{corollary}
\smallskip
\begin{proof}
Choose $k$ large enough that $\C{A}(\m{x}_k) \subset \C{A}(\m{x}^*)$
and $\bar{\m{x}}_k$ is nonempty.
Since $|\C{A}_+(\m{x}^*)| = n$ and the active constraint gradients are
linearly independent, we have $\bar{\m{x}}_k = \m{x}^*$.
By Lemma~\ref{ratio}, we must have $\m{x}_k = \m{x}^*$ whenever
$\|\m{x}_k - \m{x}^*\| < 1/c$.
\end{proof}
\smallskip

To complete the analysis of PASA in the degenerate case and show
that PASA ultimately performs only iterations in phase two, we also need
to assume that the strong second-order sufficient optimality condition holds.
Recall that a stationary point $\m{x}^*$ of (\ref{P}) satisfies the
strong second-order sufficient optimality condition if there exists
$\sigma > 0$ such that
\begin{equation} \label{opt2}
\m{d} \tr \nabla^2 f(\m{x}^*) \m{d} \ge \sigma \|\m{d}\|^2  \quad
\mbox{whenever} \quad (\m{Ad})_i = 0 \mbox{ for all } i \in \C{A}_+(\m{x}^*).
\end{equation}
First, we observe that under this assumption, the distance from
$\m{x}_k$ to $\m{x}^*$ is bounded in terms of $E(\m{x}_k)$.

\begin{lemma}\label{stability}
If $f$ is twice continuously differentiable in a neighborhood of
a local minimizer $\m{x}^*$ for $(\ref{P})$ where
the active constraint gradients are linearly independent
and the strong second-order sufficient optimality condition holds,
then for some $\rho > 0$ and for all $\m{x} \in \C{B}_\rho(\m{x}^*)$, we have
\begin{equation}\label{loc-bound}
\|\m{x} - \m{x}^*\| \le
\left[ \sqrt{1 + \left( \frac{2(1 + \kappa)(3+\kappa)}{\sigma} \right)^2 }
\; \right] E(\m{x}) ,
\end{equation}
where $\kappa$ is a Lipschitz constant for $\nabla f$ on
$\C{B}_{\rho}(\m{x}^*)$.
\end{lemma}

\begin{proof}
By the continuity of the second derivative of $f$,
it follows from (\ref{opt2}) that for $\rho> 0$ sufficiently small,
\begin{eqnarray}
\quad \quad \quad (\m{x}-\m{x}^*)\tr (\m{g} (\m{x})-\m{g} (\m{x}^*)) &=&
(\m{x}-\m{x}^*)\tr \int_0^1 \nabla^2 f (\m{x}^* + t (\m{x} - \m{x}^*)) dt \;
(\m{x} - \m{x}^*) \nonumber \\
&\ge& 0.5\sigma \|\m{x}-\m{x}^*\|^2
\label{31}
\end{eqnarray}
for all $\m{x} \in \C{B}_{\rho}(\m{x}^*) \cap \C{S}_+$, where
\[
\C{S}_+ = \{ \m{x} \in \mathbb{R}^n :
(\m{Ax}- \m{b})_i = 0 \mbox{ for all } i \in \C{A}_+(\m{x}^*)\} .
\]
Given $\m{x} \in \C{B}_\rho (\m{x}^*)$,
define $\hat{\m{x}} = \C{P}_{\C{S}_+}(\m{x})$.
Since $\C{P}_{\Omega \cap \C{S}_+}(\m{x} - \m{g}(\m{x})) \in \C{S}_+$,
it follows that
\begin{equation}\label{151}
\|\hat{\m{x}} - \m{x}\| = \|\C{P}_{\C{S}_+}(\m{x}) - \m{x}\| \le
\|\C{P}_{\Omega \cap \C{S}_+}(\m{x}- \m{g}(\m{x})) - \m{x}\| .
\end{equation}
Since $\Lambda_i(\m{x}^*) > 0$ for all $i \in \C{A}_+(\m{x}^*)$,
it follows from Corollary~\ref{SetLip} and complementary slackness
that $\rho$ can be chosen smaller if necessary to ensure that
\begin{equation}\label{*}
(\m{Ay}(\m{x}, 1) - \m{b})_i = 0 \mbox{ for all } i \in \C{A}_+(\m{x}^*) ,
\end{equation}
which implies that $\m{y}(\m{x}, 1) \in \C{S}_+$.
Since $\m{y}(\m{x}, 1) = \C{P}_\Omega (\m{x} - \m{g}(\m{x}))$ and
$\m{y}(\m{x},1) \in \C{S}_+$, we also have
$\C{P}_\Omega (\m{x} - \m{g}(\m{x})) =$
$\C{P}_{\Omega\cap\C{S}_+} (\m{x} - \m{g}(\m{x}))$.
With this substitution in (\ref{151}), we obtain
\begin{equation}\label{h53}
\|\hat{\m{x}} - \m{x}\| \le 
\| \C{P}_{\Omega} (\m{x} - \m{g}(\m{x})) - \m{x}\| =
\| \m{y}(\m{x}, 1) - \m{x}\| = \| \m{d}^1(\m{x}) \| = E(\m{x}).
\end{equation}
By the Lipschitz continuity of $\m{g}$, (\ref{h53}), and (\ref{lipd}),
it follows that
\begin{eqnarray}
\| \m{d}^1(\hat{\m{x}}) \| & \le&
\| \m{d}^1(\m{x}) \| + \|  \m{d}^1(\hat{\m{x}}) - \m{d}^1(\m{x}) \| \nonumber \\
 & \le& \| \m{d}^1(\m{x}) \| +
(2+\kappa) \| \m{x} - \hat{\m{x}} \|  \nonumber \\
&\le&  (3+\kappa)  \| \m{d}^1(\m{x}) \| \label{Z30}
\end{eqnarray}
for all $\m{x} \in \C{B}_{\rho}(\m{x}^*)$.
Since $\hat{\m{x}} = \C{P}_{\C{S}_+}(\m{x})$,
the difference $\hat{\m{x}} - \m{x}$ is orthogonal to
$\C{N}(\m{A}_{\C{I}})$ when $\C{I} = \C{A}_+(\m{x}^*)$.
Since $\hat{\m{x}} - \m{x}^* \in \C{N}(\m{A}_{\C{I}})$, it follows from
Pythagoras that
\begin{equation}\label{xbar23}
\|\m{x} - \hat{\m{x}} \|^2 + \|\hat{\m{x}} - \m{x}^*\|^2 =
\|\m{x} - \m{x}^*\|^2.
\end{equation}
Consequently, $\hat{\m{x}} \in  \C{B}_{\rho}(\m{x}^*)$
for all $\m{x} \in \C{B}_{\rho}(\m{x}^*)$, and
\begin{equation}\label{x-x*}
\|\m{x} - \m{x}^*\| =
\sqrt {\|\m{x} - \hat{\m{x}} \|^2 + \|\hat{\m{x}} - \m{x}^*\|^2} .
\end{equation}
By P8 in \cite{hz05a}, (\ref{31}), and (\ref{Z30}), we have
\begin{equation}\label{98}
\|\hat{\m{x}} - \m{x}^*\| \le
\left( \frac{1+\kappa}{0.5\sigma} \right) \|\m{d}^1 (\hat{\m{x}})\| \le
\left( \frac{(1+\kappa)(3+\kappa)}{0.5\sigma} \right) \|\m{d}^1 (\m{x})\| .
\end{equation}
Insert
(\ref{h53}) and (\ref{98}) in (\ref{x-x*}) to complete the proof.
\end{proof}


We now examine the asymptotic behavior of the
undecided index set $\C{U}$.
\begin{lemma}
\label{undecided-lemma}
If $f$ is twice continuously differentiable in a neighborhood of
a local minimizer $\m{x}^*$ for $(\ref{P})$ where
the active constraint gradients are linearly independent
and the strong second-order sufficient optimality condition holds,
and if PASA generates an infinite sequence of iterates converging to
$\m{x}^*$, then $\C{U}(\m{x}_k)$ is empty for $k$ sufficiently large.
\end{lemma}
\begin{proof}
If $E(\m{x}_k) = 0$ for some $k$, then PASA terminates and the lemma holds
trivially.
Hence, assume that $E(\m{x}_k) \ne 0$ for all $k$.
To show  $\C{U}(\m{x}_k)$ is empty for some $k$, we must show that
either
\[
\mbox{(a) } \g{\lambda}_i(\m{x}_k)/E(\m{x}_k)^{\gamma} < 1 \quad \mbox{or}\quad
\mbox{(b) } (\m{b} - \m{Ax})_i/E(\m{x})^{\beta} < 1
\]
for each $i$.
If $i \in \C{A}_+(\m{x}^*)$, then $[\m{b} - \m{A}\bar{\m{x}}_k]_i = 0$, and
by Lemma~\ref{ratio},
\[
[\m{b} -\m{Ax}_k]_i = [(\m{A}(\bar{\m{x}}_k - \m{x}_k)]_i \le
\|\m{A}\| \|\bar{\m{x}}_k - \m{x}_k\| \le
c\|\m{A}\|\|\m{x}_k - \m{x}^*\|^2.
\]
By Lemma \ref{stability}, there exists a constant $d$ such that
$\|\m{x} - \m{x}^*\| \le d E(\m{x})$ for $\m{x}$ near $\m{x}^*$.
Hence, for all $i \in \C{A}_+(\m{x}^*)$ and $k$ sufficiently large, we have
\[
[\m{b} -\m{Ax}_k]_i \le cd^2 \|\m{A}\| E(\m{x}_k)^2 =
cd^2  \|\m{A}\| E(\m{x}_k)^{2-\beta} E(\m{x}_k)^\beta .
\]
Since $\beta \in (1,2)$, $E(\m{x}_k)^{2-\beta}$ tends to zero as $k$
tends to infinity, and (b) holds when $k$ is large enough that
$cd^2 \|\m{A}\| E(\m{x}_k)^{2-\beta} < 1$.

If $i \in \C{A}_+(\m{x}^*)^c$, then $\lambda_i(\m{x}^*) = 0$.
By Corollary~\ref{SetLip}, there exist $c \in \mathbb{R}$ such that
\[
{\lambda}_i (\m{x}_k) =
{\lambda}_i (\m{x}_k) - {\lambda}_i (\m{x}^*) \le c\|\m{x}_k - \m{x}^*\|
\le cd E (\m{x}_k) = cd E(\m{x}_k)^{1-\gamma} E(\m{x}_k)^\gamma.
\]
Since $\gamma \in (0,1)$, $E(\m{x}_k)^{1-\gamma}$ tends to zero as
$k$ tends to infinity, and (a) holds when $k$ is large enough that
$cd E(\m{x}_k)^{1-\gamma} < 1$.
In summary, for $k$ sufficiently large, (a) holds when
$i \in \C{A}_+(\m{x}^*)^c$ and (b) holds when $i \in \C{A}_+(\m{x}^*)$.
This implies that $\C{U}(\m{x}_k)$ is empty for $k$ sufficiently large.
\end{proof}

As shown in the proof of Lemma~\ref{undecided-lemma},
(b) holds for $i \in \C{A}_+(\m{x}^*)$ and $k$ sufficiently large.
This implies that the constraint violation
$(\m{b} - \m{Ax})_i$ tends to zero faster than the error $E(\m{x}_k)$.
The following result, along with  Lemma~\ref{undecided-lemma},
essentially implies that PASA eventually performs only phase two.
\smallskip

\begin{lemma}
\label{mu-bound}
If $f$ is twice continuously differentiable in a neighborhood of
a local minimizer $\m{x}^*$ for $(\ref{P})$ where
the active constraint gradients are linearly independent
and the strong second-order sufficient optimality condition holds,
and if PASA generates an infinite sequence of iterates converging to
$\m{x}^*$, then there exists $\theta^* > 0$ such that
\begin{equation}\label{facecond}
e(\m{x}_k) \ge \theta^* E(\m{x}_k)
\end{equation}
for $k$ sufficiently large.
\end{lemma}
\smallskip

\begin{proof}
Let $\C{I} := \C{A}_+(\m{x}^*) \cup \C{A}(\m{x}_k)$.
The projection $\bar{\m{x}}_k$ has the property that the difference
$\m{x}_k - \bar{\m{x}}_k$ is orthogonal to $\C{N}(\m{A}_\C{I})$.
Choose $k$ large enough that
$\C{A}(\m{x}_k) \subset \C{A}(\m{x}^*)$.
It follows that $\bar{\m{x}}_k - \m{x}^* \in \C{N}(\m{A}_\C{I})$.
Hence, by Pythagoras, we have
\[
\| \m{x}_k -\bar{\m{x}}_k\|^2 +  \| \bar{\m{x}}_k - \m{x}^* \|^2
= \| \m{x}_k -\m{x}^*\|^2.
\]
Consequently,
\begin{equation}\label{201}
\|\bar{\m{x}}_k -\m{x}^*\| \le \| \m{x}_k -\m{x}^*\|,
\end{equation}
and $\bar{\m{x}}_k$ approaches $\m{x}^*$ as $k$ tends to infinity.
Choose $\rho > 0$ small enough that $f$ is twice
continuously differentiable in $\C{B}_\rho(\m{x}^*)$, and let
$\kappa$ be the Lipschitz constant for $\nabla f$ in $\C{B}_\rho(\m{x}^*)$.
Choose $k$ large enough that $\m{x}_k \in \C{B}_\rho(\m{x}^*)$.
By (\ref{201}) $\bar{\m{x}}_k \in \C{B}_\rho(\m{x}^*)$.
Since $\m{d}^1(\m{x}^*)=\m{0}$, it follows from (\ref{lipd}) that
\begin{eqnarray}
\|\m{d}^1(\m{x}_k)\| &\le& \|\m{d}^1(\m{x}_k) - \m{d}^1(\bar{\m{x}}_k)\| +
\|\m{d}^1(\bar{\m{x}}_k) - \m{d}^1(\m{x}^*)\|
\nonumber \\
&\le& (2+\kappa)(\|\m{x}_k - \bar{\m{x}}_k \| + \|\bar{\m{x}}_k - \m{x}^* \|)
\label{Z39}
\end{eqnarray}
Lemma~\ref{ratio} gives
\begin{equation}\label{999}
\|\bar{\m{x}}_k - \m{x}_k\| \le
c \|\m{x}_k - \m{x}^*\|^2
\le c \|\m{x}_k - \m{x}^*\|(\|{\m{x}}_k - \bar{\m{x}}_k\| +
\|\bar{\m{x}}_k - \m{x}^*\|) .
\end{equation}
Since $\m{x}_k$ converges to $\m{x}^*$, it follows from (\ref{999})
that for any $\epsilon > 0$,
\begin{equation}\label{xxzz}
\|\bar{\m{x}}_k - \m{x}_k\| \le \epsilon \|\bar{\m{x}}_k - \m{x}^*\|
\end{equation}
when $k$ is sufficiently large.
Combine (\ref{Z39}) and (\ref{xxzz}) to obtain
\begin{equation}\label{first}
\|\m{d}^1 (\m{x}_k) \| \le c \|\bar{\m{x}}_k - \m{x}^*\|
\end{equation}
for some constant $c$ and any $k$ sufficiently large.

Choose $\rho > 0$ small enough that (\ref{31}) holds for
all $\m{x} \in \C{B}_\rho (\m{x}^*)$, and
choose $k$ large enough that $\bar{\m{x}}_k \in \C{B}_\rho (\m{x}^*)$.
The bound (\ref{31}) yields
\begin{equation} \label{upper}
0.5\sigma\|\bar{\m{x}}_k-\m{x}^*\|^2 \le
(\bar{\m{x}}_k - \m{x}^*)\tr (\m{g}(\bar{\m{x}}_k) - \m{g}(\m{x}^*)) .
\end{equation}
By the first-order optimality conditions for a local minimizer $\m{x}^*$ of
(\ref{P}), there exists a multiplier $\g{\lambda}^* \in \mathbb{R}^m$ such that
\begin{equation}\label{1st}
\m{g}(\m{x}^*) + \m{A}\tr \g{\lambda}^* = 0 \quad \mbox{where} \quad
(\m{b} - \m{Ax}^*)\tr \g{\lambda}^* = 0 \; \mbox{ and } \;
\g{\lambda}^* \ge \m{0}.
\end{equation}
Observe that $\lambda_i^* [\m{A}(\bar{\m{x}}_k - \m{x}^*)]_i = 0$
for each $i$ since
$[\m{A}(\bar{\m{x}}_k - \m{x}^*)]_i = 0$
when $i \in \C{A}_+(\m{x}^*)$, while
$\lambda_i^* = 0$ when $i \in \C{A}_+(\m{x}^*)^c$.
Hence, we have
\[
[\m{A}(\bar{\m{x}}_k - \m{x}^*)]\tr \g{\lambda}^* = 0.
\]
We utilize this identity to obtain
\begin{equation}\label{x1}
(\bar{\m{x}}_k - \m{x}^*)\tr \m{g}(\m{x}^*) =
(\bar{\m{x}}_k - \m{x}^*)\tr (\m{g}(\m{x}^*) + \m{A}\tr \g{\lambda}^*) = \m{0}
\end{equation}
by the first equality in (\ref{1st}).

The first-order optimality conditions for the minimizer
$\m{g}^\C{I}(\bar{\m{x}}_k)$ in (\ref{dL})
imply the existence of $\g{\lambda}_\C{I}$ such that
\begin{equation}\label{2nd}
\m{g}^\C{I}(\bar{\m{x}}_k) -
\m{g}(\bar{\m{x}}_k) + \m{A}_\C{I}\tr \g{\lambda}_\C{I} = \m{0}
\quad \mbox{where} \quad \m{A}_\C{I} \bar{\m{x}}_k = \m{b}_\C{I} .
\end{equation}
Since $\C{A}(\m{x}_k) \subset \C{A}(\m{x}^*)$, we have
$\m{A}_\C{I} (\bar{\m{x}}_k - \m{x}^*) = \m{0}$,
$[\m{A}_\C{I} (\bar{\m{x}}_k - \m{x}^*)]\tr\g{\lambda}_\C{I} = \m{0}$, and

\begin{equation}\label{x2}
(\bar{\m{x}}_k - \m{x}^*)\tr \m{g}(\bar{\m{x}}_k) =
(\bar{\m{x}}_k - \m{x}^*)\tr (\m{g}(\bar{\m{x}}_k) -
\m{A}_\C{I}\tr \g{\lambda}_\C{I}) = 
(\bar{\m{x}}_k - \m{x}^*)\tr
\m{g}^\C{I}(\bar{\m{x}}_k)
\end{equation}
by (\ref{2nd}).
Combine (\ref{upper}), (\ref{x1}), and (\ref{x2}) to obtain
\begin{equation}\label{dbar}
0.5\sigma\|\bar{\m{x}}_k-\m{x}^*\| \le \|\m{g}^\C{I}(\bar{\m{x}}_k)\| .
\end{equation}

If $\C{J}$ denotes $\C{A}(\m{x}_k)$, then
$\C{J} \subset \C{I} = \C{A}(\m{x}_k) \cup \C{A}_+(\m{x}_k)$.
Hence, $\C{N}(\m{A}_\C{I}) \subset \C{N}(\m{A}_\C{J})$.
It follows that
\begin{equation}\label{p1}
\|\m{g}^\C{I} (\m{x}_k)\| \le
\|\m{g}^\C{J} (\m{x}_k)\| = e(\m{x}_k) .
\end{equation}
Since the projection on a convex set is nonexpansive,
\begin{equation}\label{p2}
\|\m{g}^\C{I}(\bar{\m{x}}_k)) -
\m{g}^\C{I}({\m{x}}_k) \| \le
\|\m{g}(\bar{\m{x}}_k) - \m{g}({\m{x}}_k) \| \le
\kappa \|\bar{\m{x}}_k - {\m{x}}_k\| .
\end{equation}
Combine (\ref{xxzz}), (\ref{p1}), and (\ref{p2}) to get
\begin{eqnarray*}
\|\m{g}^\C{I}(\bar{\m{x}}_k)\|
&\le&
\|\m{g}^\C{I}(\bar{\m{x}}_k) - \m{g}^\C{I}({\m{x}}_k)\| +
\|\m{g}^\C{I}({\m{x}}_k)\| \\
&\le& e(\m{x}_k) + \kappa \|\bar{\m{x}}_k - {\m{x}}_k\| \le
e(\m{x}_k) + \epsilon \kappa \|\bar{\m{x}}_k - {\m{x}}^*\|.
\end{eqnarray*}
Consequently, by (\ref{dbar}) we have
$0.4\sigma \|\bar{\m{x}}_k - \m{x}^*\| \le e(\m{x}_k)$ for
$\epsilon$ sufficiently small and $k$ sufficiently large.
Finally, (\ref{first}) completes the proof.
\end{proof}
\smallskip

By the analysis of Section~\ref{nondegenerate},
(\ref{facecond}) holds with $\theta^* = 1$ for a nondegenerate problem;
neither the strong second-order sufficient optimality condition nor
independence of the active constraint gradients are needed in this case.

We now show that within a finite number of iterations,
PASA will perform only LCO.
\begin{theorem}
\label{UA-rules}
If PASA with $\epsilon = 0$ generates an infinite sequence of iterates
converging to a local minimizer $\m{x}^*$ of $(\ref{P})$
where the active constraint gradients are linearly independent and the
strong second-order sufficient optimality condition holds,
and if $f$ is twice continuously differentiable near $\m{x}^*$,
then within a finite number of iterations, only phase two is executed.
\end{theorem}
\begin{proof}
By Lemma~\ref{undecided-lemma}, the undecided index set $\C{U}(\m{x}_k)$
is empty for $k$ sufficiently large, and by Lemma~\ref{mu-bound},
there exists $\theta^* > 0$ such that $e(\m{x}_k) \ge \theta^*E(\m{x}_k)$.
If $k$ is large enough that $\C{U}(\m{x}_k)$ is empty, then in phase one,
$\theta$ will be reduced until $\theta \le \theta^*$.
Once this holds, phase one branches to phase two and phase two cannot
branch to phase one.
\end{proof}

Similar to Theorem~4.2 of \cite{hz05a},
when $f$ is a strongly convex quadratic
and LCO is based on a projected conjugate gradient method,
Theorem~\ref{UA-rules} implies that when the active constraint gradients
are linearly independent, PASA converges to the optimal
solution in a finite number of iterations.
\section{Conclusions}
\label{conclusions}
A new active set algorithm PASA was developed for solving polyhedral
constrained nonlinear optimization problems.
Phase one of the algorithm is the gradient projection
algorithm, while phase two is any algorithm for linearly
constrained optimization (LCO)
which monotonically improves the value of the objective function,
which never frees an active constraint, and which has the property
that the projected gradients tend to zero,
at least along a subsequence of the iterates.
Simple rules were given in Algorithm~\ref{PASA} for branching between
the two phases.
Global convergence to a stationary point was established, while asymptotically,
within a finite number of iterations, only phase two is performed.
For nondegenerate problems, this result follows almost immediately, while
for degenerate problems, the analysis required linear independence of
the active constraint gradients, the strong second-order sufficient
optimality conditions, and a special startup procedure for LCO.
The numerical implementation and performance of PASA for general
polyhedral constrained problems will be studied in a separate paper.
Numerical performance for bound constrained optimization
problems is studied in \cite{hz05a}.
\bibliographystyle{siam}

\end{document}